\documentclass[11pt]{article}

\usepackage{header}
\numberwithin{equation}{subsection}
\begin{document}
\title{On a Question of Poltoratski}
\author{Netanel Levi\footnote{Department of Mathematics, University of California, Irvine. Supported by NSF DMS-2052899, DMS-2155211, and Simons 896624. Email: netanell@uci.edu}}
\date{}
\maketitle
\sloppy
\begin{abstract}
We study half-line discrete Schr\"odinger operators and their rank-one perturbations. We establish certain continuity and stability properties of the Fourier transform of the associated spectral measures. Using these results, we construct a sparse potential whose essential spectrum contains an open interval, and show that for every rank-one perturbation the corresponding spectral measure is non-Rajchman. This resolves a question posed in \cite{Pol}.
\end{abstract}

\section{Introduction}
In this work, we consider discrete Schr\"odinger operators on the half–line. Given a potential $V:\mathbb{N}\to\mathbb{R}$ let
\[
(H\psi)(n)=
\begin{cases}
\psi(n-1)+\psi(n+1)+V(n)\psi(n),& n>1,\\[2pt]
\psi(2)+V(1)\psi(1),& n=1,
\end{cases}
\]
act on $\ell^2(\mathbb{N})$. For $\lambda\in\mathbb{R}$ we consider the rank–one perturbation
\[
H_\lambda:=H+\lambda\langle \delta_1,\cdot\rangle\delta_1.
\]
It is well known that for every $\lambda\in\mathbb{R}$, $H_\lambda$ is essentially self-adjoint (see, e.g., \cite{Ber}). Since rank–one perturbations are compact, by Weyl’s theorem the essential spectrum is independent of $\lambda$ (see, e.g., \cite{Kato}).

The behavior of spectral properties under rank–one perturbations has been studied extensively; see, e.g., \cite{DF,DJLS,Gor,Kato,Pol,Simon,SimonWolf}. In particular, the following was proved in \cite{Gor} (see also \cite{DJMS}).
\begin{theorem}\label{gordon_thm}
There exists a dense $G_\delta$ set $Z\subseteq\mathbb{R}$ such that, for every $\lambda\in Z$, $H_\lambda$ has no eigenvalues inside the essential spectrum.
\end{theorem}

Theorem~\ref{gordon_thm} has many interesting applications. For instance, it contrasts the almost–sure pure point spectrum of the Anderson model by showing that, even in that case, \emph{for generic $\lambda$} there is no point spectrum inside the essential spectrum (and since in that model the essential spectrum contains an interval, this excludes pure point spectrum; see the discussion below). Theorem~\ref{gordon_thm} is also related to results showing that, generically, half–line Schr\"odinger operators have singular continuous spectrum \cite{DJMS,Simon2}.

Another recent application is that, in certain cases, it can be used to prove Cantor spectrum \cite{KPS} for certain operators. Concretely, an interesting consequence of Theorem~\ref{gordon_thm} is the following: if for every $\lambda\in\mathbb{R}$ the operator $H_\lambda$ has pure point spectrum, then the essential spectrum must be a Cantor set. Indeed, for generic $\lambda$ there are no eigenvalues inside the essential spectrum, while the essential spectrum is obtained as the set of limit points of eigenvalues. Hence, the essential spectrum cannot contain intervals. Since in general the essential spectrum does not have isolated points, it follows that it is a perfect nowhere–dense set (hence a Cantor set).

Theorem~\ref{gordon_thm} and its consequences can also be phrased in measure–theoretic terms. For each fixed $\lambda$, the vector $\delta_1$ is cyclic for $H_\lambda$. Thus every spectral measure is absolutely continuous with respect to the spectral measure $\mu_\lambda$ of $\delta_1$, and questions of spectral continuity reduce to the study of $\mu_\lambda$. In particular, Theorem~\ref{gordon_thm} is equivalent to the statement that, for a dense $G_\delta$ set $Z\subseteq\mathbb{R}$, the restriction of $\mu_\lambda$ to the essential spectrum has no atoms for every $\lambda\in Z$.

It is natural to ask whether the pure point property in Theorem~\ref{gordon_thm} can be replaced by other properties of the spectrum or of spectral measures. It is known that one \emph{cannot} replace it by “the spectral measure has Hausdorff dimension zero” \cite{DJLS} or by “the spectral measure has packing dimension zero” \cite{Levi}.

Given a finite Borel measure $\mu$ on $\mathbb{R}$, we call it \emph{Rajchman} if its Fourier transform decays, namely if
\[
\widehat{\mu}(t)=\int e^{-itE}\,d\mu(E)\longrightarrow 0\qquad\text{as }|t|\to\infty;
\]
otherwise, we call it \emph{non–Rajchman}. Absolutely continuous measures are Rajchman by the Riemann–Lebesgue lemma, whereas any measure with an atom is non–Rajchman. In the singular continuous realm, both behaviors are possible (for a survey on Rajchman measures, see \cite{Lyons}).

In the context of Schr\"odinger operators, by the spectral theorem, the Fourier transform is given by
\[
\widehat{\mu}(t)\;=\;\langle\delta_1,e^{-itH}\delta_1\rangle,
\]
and $\left|\widehat{\mu}\left(t\right)\right|^2$ is also known as the survival probability. This and related averaged quantities have been studied in many works, see e.g.\ \cite{DJLS,DJLS2,DT,DT2,GJ,JK,JSB,KKL,Last,Pol}. In particular, the stability of the Rajchman property under perturbations was investigated in \cite{Pol}, where it was shown that, in general, no positive uniform statement can be made. Since pure point spectrum forces non–Rajchman while Theorem~\ref{gordon_thm} excludes pure point spectrum generically whenever the essential spectrum contains an interval, the following natural question was asked in \cite{Pol}: does there exist a half–line Schr\"odinger operator whose essential spectrum contains an interval and such that, for \emph{all} $\lambda\in\mathbb{R}$, the measure $\mu_\lambda$ is non–Rajchman?

In the present work we provide a positive answer to this question. Concretely, we prove:

\begin{theorem}\label{main_thm}
 There exists a Schr\"odinger operator $H:\ell^2(\mathbb{N})\to\ell^2(\mathbb{N})$ such that
 \begin{enumerate}
     \item $\sigma_{\mathrm{ess}}(H)= [-2,2]$;
     \item for every $\lambda\in\mathbb{R}$, the spectral measure $\mu_\lambda$ of $\delta_1$ with respect to $H_\lambda:=H+\lambda\langle\delta_1,\cdot\rangle\delta_1$ is non-Rajchman.
 \end{enumerate}
\end{theorem}

Our construction produces a Schr\"odinger operator with a \emph{sparse} potential. The idea is as follows. If we allow $V$ to take the value $+\infty$ at some site, the operator becomes decoupled and, on each finite block, the survival amplitude $\widehat{\mu}_\lambda(t)=\langle \delta_1,e^{-itH_\lambda}\delta_1\rangle$ is almost periodic with large values at arbitrarily large times; hence, for all $\lambda$, $\mu_\lambda$ is non-Rajchman. We then approximate such “infinite barriers” by large finite values to retain large values of $\widehat{\mu}_\lambda$ at a certain fixed set of times. Using a variation of Duhamel’s formula (time-ordered/Dyson expansion for a bounded perturbation), we show that if one keeps the potential equal to zero on a long interval between nonzero sites, the effect on previously chosen times is negligible. Iterating this procedure yields the desired sparse potential and proves Theorem~\ref{main_thm}.
\begin{remark}
    Theorem \ref{main_thm} provides new examples of non-Rajchman singular continuous measures. Indeed, by Theorem \ref{gordon_thm}, for a dense $G_\delta$ set of parameters the spectral measure $\mu_\lambda$ is singular continuous, while Theorem \ref{main_thm} shows that it is non-Rajchman. The existence of such measures is by now well known (see \cite{Lyons} and references therein), but our construction contributes a natural operator-theoretic family to this collection.
\end{remark}

It remains an interesting problem whether the pure–point property in Theorem~\ref{gordon_thm} can be weakened to other spectral or measure–theoretic properties.

The rest of the paper is structured as follows. In Section~2 we present preliminaries and technical lemmas. In Section~3 we carry out the construction and prove Theorem~\ref{main_thm}.

\section{Preliminaries and technical lemmas}

Throughout, $\Delta$ denotes the discrete Laplacian on $\ell^2(\mathbb{N})$ with Dirichlet boundary condition, namely
\begin{center}
    $\left(\Delta\psi\right)\left(n\right)=\begin{cases}
        \psi\left(2\right) & n=1,\\
        \psi\left(n-1\right)+\psi\left(n+1\right) & n>1
    \end{cases}.$
\end{center}
We will also denote by $\delta_1$ the delta function at site $1$, namely $\delta_1\left(1\right)=1$ and $\delta_1\left(n\right)=0$ for $n\neq 1$. For a potential $V:\mathbb{N}\to\overline{\mathbb{R}}:=\mathbb{R}\cup\{\infty\}$ we set $H_V:=\Delta+V$ acting on
\[
\mathrm{dom}(H_V)=\bigl\{\psi\in\ell^2(\mathbb{N}):\ \psi(n)=0\ \text{for all }n\in V^{-1}(\infty)\bigr\}.
\]
Given $\lambda\in\mathbb{R}$ we write $H_{V,\lambda}:=H_V+\lambda\langle\delta_1,\cdot\rangle\delta_1$ and denote by $\mu_{V,\lambda}$ the spectral measure of $\delta_1$ with respect to $H_{V,\lambda}$. Its Fourier transform is
\begin{center}
    $\widehat{\mu}_{V,\lambda}(t)=\bigl\langle \delta_1, e^{-it H_{V,\lambda}}\delta_1\bigr\rangle,\qquad t\in\mathbb{R}.$
\end{center}
We will also denote $\Delta_\lambda\coloneqq\Delta+\lambda\langle\delta_1,\cdot\rangle\delta_1$.

\subsection{The essential spectrum of operators with sparse potential}
We will use the following, which can be found in \cite[Lemma~2.11]{Levi} (see also \cite[Theorem~1.3]{JL1}).
\begin{lemma}\label{lemma_ess_spec}
		Let $\left(n_k\right)_{k=1}^{\infty}$ be a strictly increasing sequence of natural numbers which satisfies $\lim_{k\to\infty}(n_k-n_{k-1})=\infty$. Let $V:\mathbb{N}\to\mathbb{R}$ and suppose that for every $n\notin\left\{n_k:k\in\mathbb{N}\right\}$, $V\left(n\right)=0$ and that $V\left(n_k\right)\to\infty$. Then for every $\lambda\in\mathbb{R}$, $\sigma_{\text{ess}}\left(H_{V,\lambda}\right)=\left[-2,2\right]$.
\end{lemma}

\subsection{Duhamel formula and its consequences}
Recall that a strongly continuous semigroup defined on a Banach space $X$ is a map $T:\mathbb{R}_{\ge 0}\to\calB(X)$ such that
\begin{enumerate}
    \item $T(0)=\text{Id}$.
    \item $T(t+s)=T(t)T(s)$.
    \item For every $x\in X$, $\|T(t)x-x\|\to 0$ as $t\to 0$.
\end{enumerate}
Every strongly continuous semigroup has a generator, namely the densely defined operator \mbox{$A:D(A)\to X$} given by
\[
Ax=\lim_{t\downarrow 0}\frac{T(t)x-x}{t}\quad(x\in D(A)),
\]
and we write formally $T(t)=e^{tA}$. If $X$ is a Hilbert space and $H$ is self-adjoint on $D(H)$, then $U(t):=e^{-itH}$, $t\in\mathbb{R}$, is a strongly continuous unitary group (Stone’s theorem). For the semigroup formulas below we substitute $A:=-\,iH$ and restrict to $t\ge 0$. We will use the following two well-known results, see e.g.\ \cite{EngelNagel,Kato}.
\begin{theorem}\emph{\cite[Corollary III.1.7]{EngelNagel}}
Consider two strongly continuous semigroups $\bigl(T(t)\bigr)_{t\ge 0}$ with generator $A$ and $\bigl(S(t)\bigr)_{t\ge 0}$ with generator $C$ on a Banach space $X$ and assume that $C=A+B$ for some bounded operator $B:X\to X$. Then
\begin{equation}\label{Duhamel_eq}
    S(t)x \;=\; T(t)x \;+\; \int_0^t T(t-s)\,B\,S(s)x\,ds
\end{equation}
holds for every $t\ge 0$ and $x\in X$.
\end{theorem}
\begin{theorem}\emph{\cite[Theorem III.1.10]{EngelNagel}}
    Under the same assumptions, we have
    \begin{equation}\label{Dyson_series_eq}
        S(t)=\sum_{n=0}^\infty S_n(t),
    \end{equation}
    where $S_0(t)=T(t)$ and
    \begin{equation}\label{Dyson_eq2}
        S_{n+1}(t)=\int_0^t T(t-s)\,B\,S_n(s)\,ds.
    \end{equation}
    Here, the series on the right-hand side of \eqref{Dyson_series_eq} converges in operator norm.
\end{theorem}
In our case, given a potential $V:\mathbb{N}\to\mathbb{R}$, we apply the above with $A=-\,iV$ and $B=-\,i\Delta_\lambda$, so that $T(t)=e^{-itV}$ and $S(t)=e^{-it(V+\Delta_\lambda)}$. Plugging this into \eqref{Dyson_series_eq} and expanding with \eqref{Dyson_eq2}, we obtain
\begin{equation}\label{Dyson_Schrodinger_eq}
    e^{-itH_{V,\lambda}}=e^{-itV}+\sum_{m=1}^\infty (-i)^m\!\!\int_{0<s_1<\cdots<s_m<t}\!\!
    e^{-i(t-s_m)V}\,\Delta_\lambda\, e^{-i(s_m-s_{m-1})V}\cdots \Delta_\lambda\, e^{-is_1V}\,ds.
\end{equation}

\subsection{Pointwise limits of (possibly infinite) potentials}
The following lemma states that the Fourier transform preserves pointwise limits when the limit potential has a single infinite barrier.

\begin{lemma}\label{lemma_pw_conv}
Let $V:\mathbb{N}\to\overline{\mathbb{R}}$ take the value $\infty$ at exactly one site $n_0$, and let $V_j:\mathbb{N}\to\mathbb{R}$ be finite-valued with $V_j(n)\to V(n)$ for every $n$. Set $H_j:=H_{V_j}$ and $H:=H_V$. Then for every fixed $\lambda\in\mathbb{R}$ and $t\in\mathbb{R}$,
\begin{center}
$\bigl\langle \delta_1, e^{-it H_{j,\lambda}}\delta_1\bigr\rangle\;\xrightarrow[j\to\infty]{}\;\bigl\langle \delta_1, e^{-it H_{\lambda}}\delta_1\bigr\rangle.$
\end{center}
\end{lemma}

We will need the following, obtained in \cite{KPS}.

\begin{proposition}\label{KPS_prop}
Let $V:\mathbb{N}\to\overline{\mathbb{R}}$ and denote $I(V):=\{n\in\mathbb{N}:V(n)=\infty\}$. Write $\ell^2(I(V))$ for the closed subspace supported on $I(V)$ and $\ell^2(I(V)^c)$ for its orthogonal complement. Define
\begin{center}
$(H_V-i)^{-1}\;:=\;\bigl((H_V|_{\ell^2(I(V)^c)}-i)^{-1}\bigr)\ \oplus\ \mathbf{0}_{\ell^2(I(V))}.$
\end{center}
If $V_j\to V$ pointwise, then
\begin{center}
$\left(H_{V_j}-i\right)^{-1}\longrightarrow\left(H_V-i\right)^{-1}$
\end{center}
strongly on $\ell^2(\mathbb{N})$.
\end{proposition}

\begin{remark}
Proposition~\ref{KPS_prop} is proved in \cite{KPS} for the Cayley transform of Schr\"odinger operators on $\ell^2(\mathbb{Z}^d)$. Their argument proceeds via strong convergence of generalized resolvents, which carries over to $\ell^2(\mathbb{N})$ without change.
\end{remark}

\begin{proof}[Proof of Lemma \ref{lemma_pw_conv}]
Fix $\lambda\in\mathbb{R}$ and $t\in\mathbb{R}$. For each $j$ define the finite-valued potentials
\begin{center}
$V_{j,\lambda}(n):=V_j(n)+\lambda\langle\delta_1,\cdot\rangle\delta_1,\qquad V_{\lambda}(n):=V(n)+\lambda\langle\delta_1,\cdot\rangle\delta_1.$
\end{center}
Then $V_{j,\lambda}\to V_{\lambda}$ pointwise and $H_{V_{j,\lambda}}=H_{j,\lambda}$, $H_{V_\lambda}=H_\lambda$. By Proposition~\ref{KPS_prop},
\begin{center}
$(H_{j,\lambda}-i)^{-1}\ \longrightarrow\ (H_{\lambda}-i)^{-1}$
\end{center}
strongly. In the single-barrier case, $(H_{\lambda}-i)^{-1}$ is the direct sum of the resolvents of the two decoupled restrictions (on $\{1,\dots,n_0-1\}$ and on $\{n_0+1,n_0+2,\dots\}$) together with $0$ on the span of $\delta_{n_0}$, so the generalized resolvent agrees with the natural decoupled resolvent. Hence, $H_{j,\lambda}\rightarrow{\ }H_\lambda$ in the strong resolvent sense. Strong resolvent convergence implies strong convergence of the unitary groups for each fixed $t$. Hence for every $\psi\in\ell^2(\mathbb{N}\setminus\left\{n_0\right\})$,
\begin{center}
$e^{-itH_{j,\lambda}}\psi\ \longrightarrow\ e^{-itH_\lambda}\psi.$
\end{center}
Taking $\psi=\delta_1$ yields
\begin{center}
$\bigl\langle \delta_1, e^{-it H_{j,\lambda}}\delta_1\bigr\rangle\ \longrightarrow\ \bigl\langle \delta_1, e^{-it H_{\lambda}}\delta_1\bigr\rangle,$
\end{center}
as required.
\end{proof}

\subsection{Fourier matrix element depends only on a finite prefix}
In this subsection, we prove that if two potentials are identical on a long prefix, then the corresponding Fourier transforms must be close for certain (early) times.

\begin{lemma}\label{lem:prefix_freezing}
Let $V,W:\mathbb{N}\to\mathbb{R}$ and set
\[
H_{V,\lambda}:=\Delta+\lambda\langle\delta_1,\cdot\rangle\delta_1+V,\qquad
H_{W,\lambda}:=\Delta+\lambda\langle\delta_1,\cdot\rangle\delta_1+W.
\]
Fix $t\in\mathbb{R}$ and $M\ge 0$. If $V(n)=W(n)$ for all $n\le N$, then for every $\lambda\in[-M,M]$,
\begin{equation}\label{eq:freeze_bound}
\sup_{|\lambda|\le M}\,
\bigl|\langle\delta_1,e^{-itH_{V,\lambda}}\delta_1\rangle-\langle\delta_1,e^{-itH_{W,\lambda}}\delta_1\rangle\bigr|
\;\le\; 2\sum_{m\ge N}\frac{\bigl((2+M)|t|\bigr)^m}{m!},
\end{equation}
and the right-hand side tends to $0$ as $N\to\infty$.
\end{lemma}

\begin{proof}
By \eqref{Dyson_Schrodinger_eq},
\begin{equation}\label{eq:dyson-phi}
e^{-itH_{V,\lambda}}
= e^{-itV}
+ \sum_{m=1}^{\infty} (-i)^m \!\!\int_{0<s_1<\cdots<s_m<t}\!\!
e^{-i(t-s_m)V} \Delta_\lambda e^{-i(s_m-s_{m-1})V}\cdots \Delta_\lambda e^{-is_1V}\,ds,
\end{equation}
and the $m$-th term has operator norm $\le \frac{(\|\Delta_\lambda\||t|)^m}{m!}$. Since $e^{-isV}$ is diagonal with diagonal entries of modulus $1$ and $\Delta_\lambda$ is a rank-one perturbation of the tridiagonal $\Delta$, the $m$-th term in the sum
$\langle\delta_1,e^{-itH_{V,\lambda}}\delta_1\rangle$ depends only on $V(1),\dots,V(m+1)$. If $V=W$ on $\{1,\dots,N\}$, then all terms with $m\le N-1$ (including $m=0$) coincide for $V$ and $W$; the difference is the tail $m\ge N$ for the two series. Using $\|\Delta_\lambda\|\le 2+|\lambda|$,
\[
\bigl|\langle\delta_1,e^{-itH_{V,\lambda}}\delta_1\rangle-\langle\delta_1,e^{-itH_{W,\lambda}}\delta_1\rangle\bigr|
\;\le\; 2\sum_{m\ge N}\frac{\bigl((2+|\lambda|)|t|\bigr)^m}{m!}
\;\le\; 2\sum_{m\ge N}\frac{\bigl((2+M)|t|\bigr)^m}{m!},
\]
which is \eqref{eq:freeze_bound}.
\end{proof}

Lemma~\ref{lem:prefix_freezing} has the following immediate corollary.
\begin{corollary}\label{prefix_cor}
    Let $V:\mathbb{N}\to\mathbb{R}$, and let $M,T>0$. Then for every $\varepsilon>0$ there exists $N\in\mathbb{N}$ such that for every $W:\mathbb{N}\to\mathbb{R}$ which is equal to $V$ on $\{1,\ldots,N\}$, for every $\lambda\in[-M,M]$ and for every $t\in(0,T)$,
    \[
        \left|\langle\delta_1,e^{-itH_{V,\lambda}}\delta_1\rangle-\langle\delta_1,e^{-itH_{W,\lambda}}\delta_1\rangle\right|<\varepsilon.
    \]
\end{corollary}

\subsection{Fourier transform for operators on a finite interval}
The goal of this subsection is to show that for operators on finite intervals and compact subsets of rank-one perturbations, one can find a finite set of times for which the corresponding spectral measures will all be bounded away from zero.
\begin{lemma}\label{lemma_fourier_compact}
Let $H$ be a Schr\"odinger operator on $\ell^2(\{1,\dots,n\})$ and let $\mu$ be the spectral measure of $\delta_1$ with respect to $H$. Then for every $\eta\in(0,1)$ and $m>0$ there exists $t>m$ such that $|\widehat{\mu}(t)|\ge \eta$.
\end{lemma}

\begin{proof}
Denote the eigenvalues of $H$ by $E_1,\ldots,E_n$ and the corresponding weights by $w_j\ge 0$ with $\sum_j w_j=1$. Then
\(
\widehat{\mu}(t)=\sum_{j=1}^n w_j e^{-itE_j},
\)
namely $\widehat{\mu}$ is an almost periodic function with $\widehat{\mu}(0)=1$. Almost periodicity implies that for any $\eta\in(0,1)$ and any $m>0$ there exists $t>m$ with $|\widehat{\mu}(t)-1|<1-\eta$, which gives the claim.
\end{proof}
Note that for fixed $V$ and $t\in\mathbb{R}$, the map $\lambda\mapsto \widehat{\mu}_{V,\lambda}(t)$ is continuous on $\mathbb{R}$. Indeed, denoting $B=\langle\delta_1,\cdot\rangle\delta_1$, by \eqref{Duhamel_eq}, we have
\begin{center}
    $e^{-it(H_V+\lambda B)}-e^{-it(H_V+\lambda' B)}
=-i(\lambda-\lambda')\!\int_0^t\! e^{-i(t-s)(H_V+\lambda B)}\,B\,e^{-is(H_V+\lambda' B)}\,ds,$
\end{center}
and so $\bigl\|e^{-it(H_V+\lambda B)}-e^{-it(H_V+\lambda' B)}\bigr\|\le |t|\,|\lambda-\lambda'|\,\|B\|$ and therefore
$|\widehat{\mu}_{V,\lambda}(t)-\widehat{\mu}_{V,\lambda'}(t)|\le |t|\,|\lambda-\lambda'|\,\|B\|$.
\begin{lemma}\label{return_times_compact_interval_lemma}
Let $V:\mathbb{N}\to\overline{\mathbb{R}}$ be such that for some $N$, $V(N)=\infty$. Then:
\begin{enumerate}
    \item For every $\lambda\in\mathbb{R}$ the measure $\mu_{V,\lambda}$ is non-Rajchman.
    \item For every $M,T>0$ there exist finitely many times $t_1,\dots,t_k>T$ with the property that for every $\lambda\in[-M,M]$ there is $i\in\{1,\dots,k\}$ such that $|\widehat{\mu}_{V,\lambda}\left(t_i\right)|\ge \frac{1}{2}$.
\end{enumerate}
\end{lemma}

\begin{proof}
\begin{enumerate}
    \item The operator acts on a finite interval, so by Lemma~\ref{lemma_fourier_compact} there are arbitrarily large $t$ with $|\widehat{\mu}_\lambda(t)|$ arbitrarily close to $1$. In particular, $\widehat{\mu}_\lambda$ does not tend to $0$.
    \item For each fixed $\lambda\in[-M,M]$, Lemma~\ref{lemma_fourier_compact} provides $t(\lambda)>T$ with $|\widehat{\mu}_\lambda\!\left(t(\lambda)\right)|\ge \tfrac{3}{4}$. For fixed $t$, the map $\lambda\mapsto \widehat{\mu}_\lambda(t)$ is continuous, hence there is an open neighborhood $U_\lambda\subseteq[-M,M]$ such that $|\widehat{\mu}_{\lambda'}(t(\lambda))|\ge \tfrac12$ for all $\lambda'\in U_\lambda$. Compactness of $[-M,M]$ yields a finite subcover $U_{\lambda_1},\dots,U_{\lambda_k}$; set $t_i:=t(\lambda_i)$.
\end{enumerate}
\end{proof}

\section{Proof of Theorem \ref{main_thm}}

\subsection*{Setup and notation}
Fix $\varepsilon\in(0,\tfrac14)$. for concreteness, we take $\varepsilon=\tfrac{1}{10}$. We construct inductively potentials $V^{(j)}:\mathbb{N}\to\mathbb{R}$ and finite sets of \emph{times} $T_j=\{t_1^{(j)},\dots,t_{k_j}^{(j)}\}\subset (j,\infty)$ with the following property:
\begin{equation}\label{eq:Pj-amp}
\text{for every }\lambda\in[-j,j]\text{ there exists }i=i^{(j)}(\lambda)\in\{1,\dots,k_j\}\text{ with }
\Bigl|\widehat{\mu}_{V^{(j)},\lambda}\!\bigl(t_{i}^{(j)}\bigr)\Bigr|\ \ge\ \frac12-\varepsilon.
\end{equation}
Moreover, $V^{(j)}$ agrees with $V^{(j-1)}$ on an initial block $\{1,\dots,N_j\}$ for some $N_j\uparrow\infty$. As part of the construction, for each $j\ge 1$ we choose $N_{j+1}>N_j$ so that for \emph{every} potential $W$ which agrees with $V^{(j)}$ on $\{1,\dots,N_{j+1}\}$ one has
\begin{equation}\label{eq:freeze-j}
\Bigl|\widehat{\mu}_{W,\lambda}(t)-\widehat{\mu}_{V^{(j)},\lambda}(t)\Bigr|<\varepsilon
\qquad\text{for all }t\in T_1\cup\cdots\cup T_j,\ \ |\lambda|\le j.
\end{equation}
In particular, \eqref{eq:freeze-j} holds for all later $V^{(m)}$ with $m>j$ and for the limit $V$ constructed below, since these agree with $V^{(j)}$ on $\{1,\dots,N_{j+1}\}$.

\medskip
Finally, we will introduce the sequence of barriers $K_j$ which will be the values the limiting potential will take at certain sites, and require that $K_j\uparrow\infty$ and $N_{j+1}-N_j\to\infty$. this will be used to show that the essential spectrum is $\left[-2,2\right]$.

\subsection*{Step 1}
Choose a site $L_1\geq 2$ and set the decoupling potential
\[
\widetilde V^{(1)}(n)=\begin{cases}
\infty,& n=L_1,\\
0,& \text{otherwise}.
\end{cases}
\]
By Lemma~\ref{return_times_compact_interval_lemma} with $M=1$ and $T=1$, there are times $T_1=\{t_1^{(1)},\dots,t_{k_1}^{(1)}\}\subset(1,\infty)$ such that for each $\lambda\in[-1,1]$ there exists $i^{(1)}(\lambda)\in\left\{1,\ldots,k_1\right\}$ with
$|\widehat{\mu}_{\widetilde V^{(1)},\lambda}(t_{i^{(1)}(\lambda)}^{(1)})|\geq \tfrac{1}{2}$.
By Lemma~\ref{lemma_pw_conv} we can approximate $\widetilde{V}^{(1)}$ by setting a large finite value at $L_1$: there exists $K_1$ such that for
\[
V^{(1)}(n)=\begin{cases}
K_1,& n=L_1,\\
0,& \text{otherwise},
\end{cases}
\]
we have
\[
\Bigl|\widehat{\mu}_{V^{(1)},\lambda}\!\bigl(t_{i^{(1)}(\lambda)}^{(1)}\bigr)\Bigr|\ge \frac12-\varepsilon,\qquad \forall\,\lambda\in[-1,1].
\]
Thus \eqref{eq:Pj-amp} holds for $j=1$, and we set $N_1:=L_1-1$.

\subsection*{Step $n+1$}
Assume \eqref{eq:Pj-amp} (with $j=n$) holds for some $n\geq 1$. Applying Corollary~\ref{prefix_cor} with $M=n$ and $T>\max\,T_1\cup\cdots\cup T_n$, we obtain an index $N_{n+1}$ so that \eqref{eq:freeze-j} holds for $j=n$, i.e.,
\[
\Bigl|\widehat{\mu}_{W,\lambda}(t)-\widehat{\mu}_{V^{(n)},\lambda}(t)\Bigr|<\varepsilon
\qquad\text{for all }t\in T_1\cup\cdots\cup T_n,\ \ |\lambda|\le n,
\]
whenever $W$ agrees with $V^{(n)}$ on $\{1,\dots,N_{n+1}\}$. We may choose $N_{n+1}$ as large as we like. In particular, we require $N_{n+1}-N_n\geq n$.

Define the truncated potential
\[
\widetilde V^{(n+1)}(l)=
\begin{cases}
V^{(n)}(l),& l\le N_{n+1},\\[2pt]
\infty,& l=N_{n+1}+1,\\[2pt]
0,& l>N_{n+1}+1.
\end{cases}
\]
By Lemma~\ref{return_times_compact_interval_lemma} with $M=T=n+1$, there are times $T_{n+1}=\{t_1^{(n+1)},\dots,t_{k_{n+1}}^{(n+1)}\}\subset(n+1,\infty)$ such that for each $\lambda\in[-(n+1),n+1]$ there exists $i^{(n+1)}(\lambda)$ with
\[
\Bigl|\widehat{\mu}_{\widetilde V^{(n+1)},\lambda}\!\bigl(t_{i^{(n+1)}(\lambda)}^{(n+1)}\bigr)\Bigr|\ge \frac12.
\]

By Lemma~\ref{lemma_pw_conv} there exists $K_{n+1}$ (which we choose to be larger than $K_n+n$) such that for
\[
V^{(n+1)}(l)=
\begin{cases}
V^{(n)}(l),& n'\le N_{n+1},\\[2pt]
K_{n+1},& l=N_{n+1}+1,\\[2pt]
0,& l>N_{n+1}+1,
\end{cases}
\]
we have, for all $t\in T_1\cup\cdots\cup T_{n+1}$ and all $|\lambda|\le n+1$,
\begin{equation}\label{eq:tilde-approx}
\left|\widehat{\mu}_{V^{(n+1)},\lambda}(t)-\widehat{\mu}_{\widetilde{V}^{(n+1)},\lambda}(t)\right|\leq\varepsilon.
\end{equation}
Combining these together we obtain that $V^{(n+1)}$ satisfies \eqref{eq:Pj-amp} for $j=n+1$ (with the time set $T_{n+1}$), while \eqref{eq:freeze-j} ensures that the values at all earlier times $T_1,\dots,T_n$ are uniformly preserved for all future modifications.

\subsection*{The limiting operator and proof of Theorem \ref{main_thm}}

By construction, for each fixed $n$ the sequence $V^{(j)}(n)$ stabilizes as $j\to\infty$. Define
\begin{center}
$V(n):=\underset{j\to\infty}{\lim}V^{(j)}(n)$.
\end{center}
\begin{proposition}\label{main_prop}
    Consider the rank-one family $\left(H_{V,\lambda}\right)_{\lambda\in\mathbb{R}}$. Then for every $\lambda\in\mathbb{R}$, we have
    \begin{enumerate}
        \item $\sigma_{\text{ess}}\left(H_{V,\lambda}\right)=\left[-2,2\right]$.
        \item $\mu_{V,\lambda}$ is non-Rajchman.
    \end{enumerate}
\end{proposition}
\begin{proof}
    \begin{enumerate}
        \item This follows from Lemma~\ref{lemma_ess_spec}: by construction the special sites $n_k:=N_k+1$ satisfy $n_{k+1}-n_k\to\infty$ and the barrier heights $K_k\to\infty$, hence $\sigma_{\text{ess}}(H_{V,\lambda})=[-2,2]$ for every $\lambda$.
        \item Since $V$ is fixed, we denote $\widehat{\mu}_{V,\lambda}\coloneqq\widehat{\mu}_\lambda$ for every $\lambda\in\mathbb{R}$. Fix $\lambda\in\mathbb{R}$ and choose $r\in\mathbb{N}$ with $r\ge |\lambda|$. For each $j\ge r$, since $V^{(j)}$ satisfies \eqref{eq:Pj-amp}, there exists $i^{(j)}(\lambda)$ with
        \[
            \Bigl|\widehat{\mu}_{V^{(j)},\lambda}\!\bigl(t_{i^{(j)}(\lambda)}^{(j)}\bigr)\Bigr|\ \ge\ \frac12-\varepsilon,
        \]
        and $t_{i^{(j)}(\lambda)}^{(j)}\in (j,\infty)$. By \eqref{eq:freeze-j} (with $W=V$ and $n=j$), we have
\[
    \Bigl|\widehat{\mu}_{\lambda}\!\bigl(t_{i^{(j)}(\lambda)}^{(j)}\bigr)-\widehat{\mu}_{V^{(j)},\lambda}\!\bigl(t_{i^{(j)}(\lambda)}^{(j)}\bigr)\Bigr|
    \ <\ \varepsilon.
\]
Therefore,
\[
    \Bigl|\widehat{\mu}_{\lambda}\!\bigl(t_{i^{(j)}(\lambda)}^{(j)}\bigr)\Bigr|
    \ \ge\ \frac12-2\varepsilon=\frac{3}{10}.
\]
Since for every $j$ we have $T_j\subseteq\left(j,\infty\right)$, the sequence $t_{i^{(j)}(\lambda)}^{(j)}$ goes to infinity as $j\to\infty$. Since $\widehat{\mu}_\lambda$ does not converge to $0$ along this sequence, $\mu_\lambda$ is non-Rajchman as required.
    \end{enumerate}
\end{proof}

\begin{proof}[Proof of Theorem \ref{main_thm}]
    Theorem~\ref{main_thm} is an immediate consequence of Proposition~\ref{main_prop}.
\end{proof}


\end{document}